\documentclass[a4paper, 12pt]{article}
\usepackage{a4}
\usepackage{amsmath}
\usepackage{dsfont}
\usepackage{amstext}
\usepackage{amssymb}
\usepackage{amsthm}
\usepackage{verbatim}
\usepackage{amscd}
\usepackage{setspace}

\usepackage{layout}

\usepackage[latin1]{inputenc}
\usepackage{eucal}

\usepackage{pstricks}
\usepackage{psfrag}

\usepackage{ifpdf}
\ifpdf
\usepackage[pdftex]{graphicx}
\DeclareGraphicsExtensions{.pdf} \else
\usepackage[dvips]{graphicx}
\DeclareGraphicsExtensions{.eps} \fi

\newtheorem{teo}{Theorem}[section]
\newtheorem{defi}[teo]{Definition}
\newtheorem{lem}[teo]{Lemma}
\newtheorem{pro}[teo]{Proposition}
\newtheorem{remark}[teo]{Remark}

\newtheorem{coro}[teo]{Corollary}
\newcommand{\N}{{\rm I\!N}}
\newcommand{\R}{{\rm I\!R}}

\newcommand{\virg}[1]{``#1''}

\numberwithin{equation}{section}

\title{h-cobordism and s-cobordism Theorems:
Transfer over Semialgebraic and Nash Categories, Uniform bound and Effectiveness }

\author{Demdah Kartoue Mady}
\date{}
\begin{document}
\maketitle
\begin{abstract}
The h-cobordism theorem is a noted theorem in differential and PL topology. A generalization of the h-cobordism theorem for possibly non simply connected manifolds is the so called s-cobordism
theorem.
In this paper, we prove semialgebraic and Nash versions of these theorems.
That is, starting with semialgebraic or Nash cobordism data, we get a semialgebraic homeomorphism (respectively a Nash diffeomorphism).
The main tools used are semialgebraic triangulation and Nash approximation.

One aspect of the algebraic nature of semialgebraic or Nash objects is that one can measure their complexities. We show h and s-cobordism theorems with a uniform bound on the complexity of the semialgebraic homeomorphism (or Nash diffeomorphism) obtained in terms of the complexity of the cobordism data. The uniform bound of semialgebraic h-cobordism cannot be recursive, which gives  another example of 
non effectiveness in real algebraic geometry see [ABB].
Finally we  deduce the validity of the semialgebraic  and Nash versions of these theorems over any real closed field.
\end{abstract}
\section*{Introduction}
The h-cobordism theorem is a classical result in differential and PL topology. In this paper we prove that it holds true in semialgebraic and Nash categories over any real closed field.

Let $M$ be a compact smooth manifold having as boundary a disjoint  union of two smooth manifolds
$M_0$ and $M_1$ such that $M_0$ and $M_1$ are both deformation retracts of $M$.
 A triplet $(M, M_0, M_1)$ like this  is said to be an  \em{h-cobordism.}
\rm  The h-cobordism theorem  states:
\begin{teo} (Smale 1961)\\
Let $( M, M_0, M_1) $ be a simply connected smooth h-cobordism. If \rm  $\mbox{dim} M\geqq 6$  \em{then M is
diffeomorphic to $M_0 \times[0,1]$}.
\end{teo}

It is a general procedure to use Tarski-Seidenberg Principle to transfer statements from $\R$ to any real closed field, once uniform  bounds are found for the complexity of all the semialgebraic or Nash objects involved in the statements.

To do this, first of all we need a semialgebraic or Nash version of the h-cobordism theorem, that we easily get  triangulating our manifold and using an approximation result.

Secondly we have to make precise the meaning of some topological facts in a semialgebraic setting and verify that definitions are consistent.

The uniform bound which is established is the following: the complexity of the semialgebraic homeomorphism $f:M\rightarrow M_0\times [0,1]$ can be bounded in terms  only of the complexity of the h-cobordism $(M,M_0,M_1).$ 
This enables us  to translate  the semialgebraic h-cobordism theorem to a countable  family of first order statements of the theory of real closed fields (one for each complexity of the triplet $(M,M_0,M_1)$).

Once this is done, we can use Tarski-Seidenberg  Principle to transfer the semialgebraic or Nash h-cobordism theorem to any real closed field.

In a similar way we get also the semialgebraic and Nash s-cobordism 
theorems over any real closed field

It is a natural question to ask whether the uniform bounds that we get 
are effective or not,  that is 
to ask whether  the complexity of the isomorphism $f:M \to 
M_0 \times [0,1]$   is bounded by a recursive function of the complexity of $(M,M_0,M_1)$.

We prove that this cannot be the case for the h-cobordism theorem. The failure is because we have
to recognise wether a semialgebraic set is simply connected or not.

The non effectiveness of the  h-cobordism theorem  is another exemple of 
non effectiveness in real algebraic geometry see [ABB].

I would like to thank F. Acquistapace, F. Broglia, M. Coste and   M. Shiota  for their advices during the preparation of  this work.

\section{Semialgebraic and Nash  h-cobordism theorems and s-cobordism theorems}
We shall agree in this work that every semialgebraic mapping is continuous. A semialgebraic manifold is a semialgebraic subset $M$  of $\R^n$ (or of $R^n$, where $R$ is a real closed field) equipped with a finite semialgebraic atlas, that is, $M= \cup_{i\in I} U_i$, $I$ finite set, $U_i$  open semialgebraic in $M$ and $\phi_i:U_i\rightarrow \R^d$ a semialgebraic  homeomorphism onto an open semialgebraic subset of $\R^d$).\\
A Nash manifold is a semialgebraic  subset of $\R^n$ (or of $R^n$) which is also a $C^{\infty}$ submanifold and is equipped with a finite Nash atlas $\{U_i, \phi_i\}$ where $\phi_i$ is semialgebraic and $C^{\infty}$. For more detail see [S].

Any compact semialgebraic  set $S\subset \R^n$ (or $R^n$) can be triangulated,  i.e.  there  is a finite simplicial complex $K$ in $\R^n$ (or $R^n$) and a semialgebraic homeomorphism $h:|K|\rightarrow S$ (where $|K|$ is the union of the simplices of $K$). Moreover the semialgebraic triangulation can be chosen compatible with a finite  family $(T_i)_{i \in I}$ of a semialgebraic  subsets of $S$, which means that each $T_i$ is the image by $h$ of the union of some open simplices (see [BCR] p.217).
 The semialgebraic triangulation is unique, in the sense that any two compact polyhedra $K$ and $L$ which are semialgebraically homeomorphic are PL homeomorphic (cf.[SY]). Hence any semialgebraic set gets a unique PL structure.

Also we get:
\begin{pro}\label{sapl}
 Every compact semialgebraic manifold is semialgebraically 
  \newline
 homeomorphic to a PL manifold.
\end{pro}
\begin{proof}
Let $M$ be a compact semialgebraic manifold of dimension $m$. There is a  semialgebraic triangulation 
$h:|K|\longrightarrow M $ where  $K$ is a finite  simplicial complex.
We have to check that the polyhedron $|K|$ is a PL manifold. Take $x \in |K|$ and  $y =h(x)$. 
By  definition, 
there is a neighbourhood  $V$ of $y$ in $M$ semialgebraically homeomorphic to a open semialgebraic set $U$ 
in $\R^m$, that is, there is a semialgebraic chart  $(V, \phi)$ such that $\phi:V
\longrightarrow U$ is a semialgebraic homeomorphism. Then,
there is an open neighbourhood $h^{-1}(V)$ of $x$ in $|K|$
semialgebraically homeomorphic to an open semialgebraic set $U$ of
 $\R^m$. There is  a closed PL ball $B\subset U$ such that $\phi(y)\in \mbox{Int}B$. It follows that the set $W=h^{-1}\circ
\phi^{-1}(B)$ is a closed and bounded neighbourhood of 
$x$ in $|K|$.  Assuming the triangulation $h$ to be  compatible
with $\phi^{-1}(B)$, one has that $W$ is a polyhedron.
It follows that $W$ and $B$ are
semialgebraically homeomorphic. By uniqueness, they are PL homeomorphic. Then
$\mbox{Int}(W)$ and $\mbox{Int}(B)$ are PL
homeomorphic. This shows that $|K|$ is a  PL manifold.
\end{proof}
\begin{defi}\rm  Let $ (M, M_0, M_1)$ be a triple of compact semialgebraic manifolds such that:
 $\partial M = M_0\bigcup M_1$ and
  $ M_0\cap M_1= \emptyset .$
Then, $ ( M, M_0, M_1) $ is called a \em{semialgebraic cobordism}\rm.\\
  A semialgebraic cobordism $( M, M_0, M_1) $ is said to be a  \em{ semialgebraic h-cobordism } \rm  if the inclusions
$M_0 \hookrightarrow M$ and
 $M_1 \hookrightarrow M$ are  semialgebraic homotopy equivalences, that is, 
  the deformation retractions are semialgebraic.
\end{defi}
Let $X$ be a semialgebraic set defined over a real closed field
$R$. The semialgebraic fundamental group of $X$ can be defined considering
 semialgebraic loops and semialgebraic homotopies between loops. We write
$\pi_1(X,x_0)_{alg}$ with $x_0 \in X$.
\newline
If $R =\R$ we have:

\begin{pro}(\cite{DK}, Theorem 6.4, p.271) \label{con}\\
Let $X$ be a closed semialgebraic subset of $\R^n$. Then 
 $\pi_1(X,\,x_0)_{alg}$ and $\pi_1(X,\,x_0)$ are 
isomorphic.
\end{pro}
The results just recalled enable us to translate the PL h-cobordism theorem to the semialgebraic category.
\begin{teo}\label{hcob}
Let $( M, M_0, M_1) $ be a semialgebraic h-cobordism simply
connected in $\R^n$. If \rm $\mbox{dim} M\geqq 6$  \em{then $M$ is
semialgebraically homeomorphic to $M_0 \times [0,1]$.}
\end{teo}
\begin{proof}
By Proposition \ref{sapl}, there is a semialgebraic triangulation
$\lambda: |K| \longrightarrow M.$ 
 We may assume that the triangulation is
compatible with the submanifolds $M_0$ and $M_1$, i.e. there are  simplicial subcomplexes $K_0, K_1 \subset K$ that $\lambda(|K_0|)=M_0$ and $\lambda(|K_1|)=M_1$.  The polyhedra  $|K|$, $|K_0|$ and $|K_1|$  are  compact PL manifolds (Proposition \ref{sapl}). The polyhedra $|K_0|$ and $|K_1|$  are semialgebraic deformation retracts of $|K|$. They are also PL deformation retracts of $|K|$, by PL approximation (cf [H], Lemma 4.2, p. 92). It follows that $(|K|, |K_0|, |K_1|)$ is a  simply
connected PL h-cobordism.  Then by the PL h-cobordism theorem 
$ |K|\stackrel{PL}{\cong}|K_0|\times [0,1],$ where $\stackrel{PL}{\cong}$ indicates the PL homeomorphism. Since a compact PL manifold is a semialgebraic manifold and a PL homeomorphism 
between  compact PL manifolds is a semialgebraic homeomorphism, it follows easily that $M$ is
semialgebraically homeomorphic to $M_0 \times [0,1]$. This ends the proof.
\end{proof}

\section{Extension of some topological properties}
 In this section, we want to extend the meaning of some  topological properties as semialgebraic simple connectedness and  s-homotopy  from $\R$ to any real closed field $R$. This will be useful in the sequel of this paper.
 
 Let $R$ and $K$ be two real closed fields such that $K$ is a real closed extension of $R$. If $X$ is semialgebraic subset of $R^n$, we denote by $X_K$ the semialgebraic subset of $K^n$ defined by the same boolean combination of polynomial equation and inequalities as $X$. Actually  by Tarski-Seidenberg Principle, $X_K$ depends only on $X$ and not on its description.
\begin{pro}\label{homeo}
Let $X$ and $Y$ be two semialgebraic subsets of $R^n$. The semialgebraic sets $X$ and
$Y$ are semialgebraically homeomorphic if and only if  $X_K$
and $Y_K$ are semialgebraically homeomorphic.
\end{pro}
\begin{proof}
The first implication is obvous.\\
Conversely, set  $X= \{x\in R^n: \phi(a,x) \},$
$Y=\{x\in R^n: \psi(b,x) \}$ where $\phi(a,x)$ and $\psi(b,x)$ are first order 
formulas of the theory of real closed fields with parameters $a\in R^m$ and $b\in R^{m'}$.  Let $f$ be a semialgebraic homeomorphism
from $X_K$ onto $Y_K$. Let $\psi(c,x,y)$ be a first order 
formula of the theory of real closed field with parameter
$c\in K^{r}$ defining  
$ \Gamma_f = \{(x,y)\in K^n\times K^n: \psi(c,x,y) \}$ the graph of $f$. One can get a first order formula in the theory of real closed fields $\lambda(a,b,c)$ which says that $f$ is a semialgebraic homeomorphism between $X_K$ and $Y_K$. So, we have 
$K\models \lambda(a,b,c).$  Let us observe that: $K\models \exists z\,\lambda(a,b,z)$ with $a\in R^m$ and $b\in R^{m'}$.
By  Tarski-Seidenberg Principle, we get: $R\models \exists
z\,\lambda(a,b,z).$ That is, there exists a parameter $c'\in
R^{r}$ which defines a homeomorphism between $X$ and
$Y$. This completes the proof.
\end{proof}

Next step is to show  that the notion of being $C^r$-Nash manifold  can be translated into a  first order formula of the theory of real closed fields. This can be done using the fact that a $C^r$-manifold is locally the graph of a $C^r$-map.

\begin{pro} \label{variete}  Let $S\subset \R^n$ be a semialgebraic set. Then,
 the statement
 \virg{$S$ is a $C^r$-Nash submanifold of  $\R^n$ of dimension $m$}
  can be translated into a first order formula of the theory of real closed fields.

\end{pro}
\begin{proof} Set $x=(x_1,\ldots,x_n)$, $y=(x_1,\ldots,x_m)$,
$z=(x_{m+1},...,x_n)$. We can write a formula  $\Phi(x)$ which says that there are
 positive real numbers $\varepsilon$ and $\eta$ such that
$S\cap (B^m(y,\varepsilon)\times B^{n-m}(z,\eta))$ is the graph of a $C^r$-function
  from $B^m(y,\varepsilon)$ to
$\R^{n-m}$.  Furthermore,
for all permutation $\sigma$ of $\{1,\ldots,n\}$, let us indicate by
$\Phi_\sigma(x)$ the formula that says the same things for the image of 
$x$ and  $S$ by the permutation $\sigma$ of the coordinates (in order  to get all projections on $m$ coordinates among $n$).
There exists  a
permutation $\sigma$ of $\{1,\ldots,n\}$ such that $\Phi_\sigma(x)$
is true.
\newline
We deduce the following formula:
$\forall x\in S\ \bigvee_\sigma \Phi_\sigma(x)$
which says clearly that $S$ is a $C^r$-Nash submanifold of dimension $m$ of  $\R^n.$

\end{proof}

\begin{pro}\label{bord}
 Let $S$ and $T$ be semialgebraic subsets of $\R^n$ such that
  $T\subset S$. Then, the statement  \virg{$S$  is a $C^r$-Nash submanifold of $\R^n$ of dimension $m$,  with boundary 
 the set $T$} can be translated into a first order  formula of the theory of real closed fields.
\end{pro}
\begin{proof} Let $x=(x_1,\ldots,x_n)$, $y=(x_1,\ldots,x_{m-1})$,
$z=(x_{m+1},...,x_n)$.
 We can write a first order formula of the theory of real closed fields  $\Psi(x)$ which says that there are
 positive real numbers $\varepsilon$, $\delta$ and $\eta$ such that both 1) and 2) below hold.
\begin{enumerate}
\item $T\cap (B^{m-1}(y,\varepsilon)\times {]} x_m-\delta,x_m+\delta{[} \times
B^{n-m}(z,\eta))$ is the graph of a $C^r$- map
$g: B^{m-1}(y,\varepsilon)\longrightarrow\R^{n-m+1}$.

Denote by $\xi : B^{m-1}(y,\varepsilon)\to \R$ the first component of  the map $g$ and by $\Gamma_{\xi}^+\subset
B^{m-1}(y,\varepsilon)\times \R $, $\Gamma_{\xi}^-\subset
B^{m-1}(y,\varepsilon)\times \R $ its over and undergraph (that is 
$\Gamma_{\xi}^+=\{(u,v)\in B^{m-1}\times \R: v\geq \xi(u)  \}$ similarly $\Gamma_{\xi}^-$).

\item $S\cap(B^{m-1}(y,\varepsilon)\times {]} x_m-\delta,x_m+\delta{[} \times B^{n-m}(z,\eta))$
is the graph of a semialgebraic $C^r$-map from either $\Gamma_{\xi}^+\cap
(B^{m-1}(y,\varepsilon)\times {]} x_m-\delta,x_m+\delta{[})$ to
$\R^{n-m}$, or  $\Gamma_{\xi}^-\cap
(B^{m-1}(y,\varepsilon)\times {]} x_m-\delta,x_m+\delta{[})$.
\end{enumerate}
Further, for every permutation $\sigma$ of $\{1,\ldots,n\}$,
let us indicate by $\Psi_\sigma(x)$ the formula that says the same thing for the 
image of $x$, $S$ and $T$ by the permutation $\sigma$ of coordinates. We construct
 $\Psi_\sigma(x)$ following the same idea as in the proof of  
 Proposition \ref{variete}, and we take the conjunction of 
$\forall x\in S\setminus T\ \bigvee_\sigma \Phi_\sigma(x)$
and
$\forall x\in T\ \bigvee_\sigma \Psi_\sigma(x)$ with
$\Phi_\sigma(x)$ as in the proof of 
Proposition \ref{variete}. There is a delicate point when we say that we have the graph of a $C^r$-differentiable function over something which is not open (the overgraph). But we can take the coordinate map 
\begin{align}
g: &B^{m-1}\times \R^+ \rightarrow \Gamma_{\xi}^+
\nonumber \\
&\qquad \quad(x,y)\mapsto (x,y+\xi(x)).\nonumber
\end{align}
which identifies 
the overgraph with a half-space and  compute derivatives of this function with respect to these  coordinates , which completes the proof. 
\end{proof}
The following proposition assures  that the fundamental group of a semialgebraic set does not change during a  real closed  extension. 
\begin{pro}(\cite{DK}, Theorem 6.3, p. 270)\label{realcon}\\
Let $X$ be a semialgebraic set in  $R^n$, $x_0\in X$ and $K$
be a real closed extension of $R$.
 The map $ k: \pi_1(X, x_0 )_{alg} \longrightarrow \pi_1(X_K, x_0)_{alg}, $ defined by $k[\gamma]:= [\gamma_K]$ is a 
group  isomorphism.
\end{pro}

\section{Semialgebraic  h-cobordism theorem over any real closed field}
In this section we prove the existence of a uniform bound on the complexity of the homeomorphism in the semialgebraic h-cobordism  theorem and use this bound to transfer the semialgebraic h-cobordism  theorem over any real closed field.
\begin{defi}\rm Let $R$ be a real closed field. A semialgebraic subset of $R^n$ is said of complexity at most $(p,q)$ if it admits a description as follows
 $$ \bigcup_{i=1}^s\bigcap_{j=1}^{k_i}\{x\in R^n | f_{ij}(x) \ast_{ij}0 \},$$ where $f_{ij}\in R[X_1,...,X_n]$,  and
$\ast_{ij}\in \{<, >, =\}$, $\Sigma _{i=1}^s k_i\leq p$,
$\mbox{deg}(f_{ij}) \leq q$ for $i=1,...,s$ and $j=1,...,r_i$.\\ \rm
The \em{complexity} \rm  of a semialgebraic subset $S$ of
$R^n$ is the smallest couple $(p,q)$, with respect to the
lexicographic order, such that $S$ admits the description above.
\end{defi}
Assume a semialgebraic subset $\mathcal{S}(R)\subset R^k$ is defined for any real closed field $R$. We say that $\mathcal{S}$ is defined \em{uniformly} \rm when there is a first order formula of the theory of real closed fields without parameter which describes $\mathcal{S}(R)$ for every real closed field $R$. In order to check that $\mathcal{S}(R)$ is defined uniformly, it suffices to check that, for any real closed extension $R\subset K$, one has $\mathcal{S}(R)_K= \mathcal{S}(K)$.

Assume a semialgebraic subset $\mathcal{S}(R,n,p,q)\subset R^{\alpha(n,p,q)}$ is defined for every real closed field $R$ and any positive integers $n, p, q.$  Assume moreover that for any $n,p,q$, $\mathcal{S}(R,n,p,q)$ is uniformly defined by a formula without parameter $\Phi_{n,p,q}$: Then we say that  $\mathcal{S}$ is \em{effectively} \rm  defined if there is an algorithm which, given $n,p,q$, produces $\Phi_{n,p,q}$. (Technically, using a G\"odel numbering of formulas, this means that the function which associates to $(n,p,q)$ the G\"odel number of $\Phi_{n,p,q}$ is recursive cf. ([Ma], Chap. VII, § 4, p. 242)). In what follows, we drop the explicit dependence on  $R$ and we write  $\mathcal{S}(n,p,q)$ instead of  $\mathcal{S}(R,n,p,q)$.
\begin{pro}\label{prepa}
There exist a semialgebraic subset $\mathcal{A}(n,p,q)$  in some affine space $R^{\alpha(n,p,q)}$ and a semialgebraic family $\mathcal{S}(n,p,q)\subset\mathcal{A}(n,p,q)\times R^n$  such that:\\
 (i) For every $a\in \mathcal{A}(n,p,q)$ the fiber 
  $$\mathcal{S}_a(n,p,q)=\{x\in R^n: (a,x)\in \mathcal{S}(n,p,q)\}$$
  is a semialgebraic subset of complexity at most $(p,q)$ of $R^n$\\
  (ii) For every semialgebraic subset $S\subset R^n$ of complexity at most  $(p,q)$, there is  $a\in \mathcal{A}(n,p,q)$ such that: $S=\mathcal{S}_a(n,p,q).$\\
 $\mathcal{A}(n,p,q)$ and $\mathcal{S}(n,p,q)$ are defined in a uniform way by  first order formulas of the theory of real closed fields without parameters which can be effectively constructed from $n,p,q$.
\end{pro}

\begin{proof}

Let us first give a description of the fibers of $\mathcal{S}(n,p,q)$ which allow us to show that their union is  semialgebraic set.\\
We start with a set of $p$ polynomials of degree $\leq q$. Let us call
$f_0,...,f_{p-1}$ the polynomials. A system of sign conditions  over these polynomials is given by an element  $\sigma \in
\{-1,0,1\}^p$.  This system of signs condition   is satisfied in the set
$$\bigcap_{i=0}^{p-1}\{x\in R^n: \, \mbox{sign}(f_i(x))=\sigma_i \}.$$
We will show that a semialgebraic in $R^n$ of complexity at most $(p,q)$
can be described by a boolean combination of sign conditions over  $p$ polynomials in $n$ variables  of degree $\leq q$,
that is,  it can be written in the following form:
$$\bigcup_{\sigma\in\Sigma}\bigcap_{i=0}^{p-1}\{x\in R^n: \, \mbox{sign}(f_i(x))=\sigma_i \},$$
where $\Sigma$ is a subset of $\{-1,0,1\}^p.$\\
Let us index  the subsets of $\{-1,0,1\}^p$ by the integers 
$l$ starting form $0$ to $2^{3^p}-1.$
Now, we describe the space of parameters. To do it, we introduce 
the notation $f_a$ to indicate the  polynomial in $n$
variables of degree $\leq q$ where the list of the coefficients of the monomials 
of $f$ ordered with respect to  the lexicographic order is $a\in R^N$ with  $N=\left(
\begin{array}{c}
n+q\\
q
\end{array}
\right)$. Consider
$$(a_0,...,a_{p-1},l) \in (R^N)^p\times \{0,...,2^{3^p}-1 \}.$$
The semialgebraic set of  $R^n$ corresponding to this parameter is
$$\bigcup_{\sigma\in\Sigma[l]}\bigcap_{i=0}^{p-1}\{x\in R^n: \, \mbox{sign}(f_{a_i}(x))=\sigma_i\}.$$
 We can then describe  $\mathcal{S}(n,p,q)$ by the following formula in 
$(a_0,...,a_{p-1},l,x)\in (R^N)^p\times \{0,...,2^{3^p}-1 \}\times
R^n $:
$$\Phi_{n,p,q}(a_0,...,a_{p-1},l,x)=\bigvee_{\sigma\in\Sigma[l]}\left(\bigwedge_{i=0}^{p-1} \, \mbox{sign}(f_{a_i}(x))=\sigma_i
\right).$$ So, we have  
$$\mathcal{S}(n,p,q)=\bigcup_{l=0}^{2^{3^p}-1}\{(a_0,...,a_{p-1},l,x)\in (R^N)^p\times R\times R^n: \Phi_{n,p,q}(a_0,...,a_{p-1},l,x)\}.$$ As defined, $\mathcal{S}(n,p,q)$ is a semialgebraic subset of $(R^N)^p\times R\times R^n$. The set
$\mathcal{A}(n,p,q)=\bigcup_{l=0}^{2^{3^p}}(R^N)^p\times \{l\}\subset(R^N)^p\times R$ gives us the space of 
parameters of the semialgebraic subsets of  $R^n$ of complexity at most
$(p,q)$. Then, one 
obtains effectively for any real closed field that the space of parameters
$\mathcal{A}(n,p,q)$ is a semialgebraic subset of
$(R^N)^p\times R$.
\end{proof}

If $\mathcal{S}_a$ is the semialgebraic set parametrized by $a\in \mathcal{A}(n,p,q)$ by abuse of notation we will write $\mathcal{S}_a\in \mathcal{A}(n,p,q)$.  
Let us recall   the definition of semialgebraic trivialisation of a semialgebraic map.
\begin{defi}\rm
A continuous semialgebraic map $f:A\longrightarrow B$ 
is said to be \em{semialgebraically trivial over a semialgebraic subset $C\subset B$} \rm 
if there  is a semialgebraic set $F$ and a semialgebraic homeomorphism $h:
f^{-1}(C)\longrightarrow C\times F$, such that the composition of $h$
with the  projection $C\times F\rightarrow C$ is equal to the restriction of 
 $f$ to $f^{-1}(C)$. This is shown by the following commutative diagram: $$\begin{CD}
A\supset f^{-1}(C) @>h>> C\times F\\
@VVfV @VVpr_1V\\
B\supset C @>=>> C
\end{CD}.$$
The homeomorphism $h$ is called a \em{ semialgebraic trivialisation
} \rm of $f$ over $C$.
\newline
We say that the trivialisation $h$ is compatible with a subset 
 $D\subset A$ if there is a subset 
$G\subset F$ such that $h(D\cap
f^{-1}(C))= C\times G$.
\end{defi}
We can now state  Hardt's theorem. A detailed proof which works  over any real  closed in field can be found in  [BCR, p.221].
\begin{teo} 
Let $A\subset R^n$, $B\subset R^m$ be two semialgebraic sets and $f:A\longrightarrow B$ a semialgebraic map. There is a finite semialgebraic partition of
$B=\cup_{i=1}^k B_i$ such that  $f$ is semialgebraically  trivial
over each $B_i$. Moreover, if $A_1,...,A_h$ are finitely many semialgebraic subsets of  $A$, we can ask  each 
trivialisation $h_i$ to be compatible with all $A_j$.
\end{teo}
\begin{remark}\rm
Let $a$ and $b$ be any two elements of the same $B_i$ then, one gets that $f^{-1}(a)$ and $f^{-1}(b)$
are semialgebraically  homeomorphic.
\end{remark}

\begin{pro}\label{sdf}
Given the integers $n$, $p$ and $q$, there  exists a couple
of integers $(t,u)$ such that for every couple of semialgebraic sets of complexity at most $(p,q)$ which are 
semialgebraically homeomorphic, there is a semialgebraic 
homeomorphism $f$ between them whose graph $\Gamma_f \in
\mathcal{A}(2n,t,u)$.
\end{pro}
\begin{proof} Consider the following projection :
\begin{align}
\Pi:\, &R^{pN+1}\times R^n \rightarrow R^{pN+1}
\nonumber \\
&\qquad \quad (a,x)\mapsto a \nonumber
\end{align}
 with $a\in R^{pN+1}$ and $x\in R^n$. We have that
$\mathcal{S}(n,p,q) = \{(a,x)\in \mathcal{A}(n,p,q)\times R^n: x\in
\mathcal{S}_a\}$ where $\mathcal{S}_a$ is a semialgebraic subset of $R^n$
parametrized by  $a \in \mathcal{A}(n,p,q)$. The set
$\mathcal{S}(n,p,q)$ is a  semialgebraic subset of
$R^{pN+1}\times R^n$ (see the proof of Lemma \ref{prepa}). The
projection  $\Pi_{|\mathcal{S}(n,p,q)}:
\mathcal{S}(n,p,q)\rightarrow \mathcal{A}(n,p,q)$ is a semialgebraic map.
 By the  Hardt trivialisation theorem, applied to the  semialgebraic map
$\Pi_{|\mathcal{S}(n,p,q)}$, there exists a finite semialgebraic partition of
$\mathcal{A}(n,p,q)$ in  $S_i$:
$\mathcal{A}(n,p,q)=\bigcup\limits_{i=1}^s S_i$ such that for each $i$, there exists a semialgebraic subset
 $X_i$ and a semialgebraic homeomorphism 
  $h_i$ such that the following diagram commutes:
$$\begin{CD}
 \Pi^{-1}(S_i) @>h_i>> S_i\times X_i\\
@VV\Pi V @VVpr_1V\\
S_i @>=>> S_i
\end{CD}.$$ 
 As the number of  trivialisation homeomorphisms is finite, let us take maximum $(u,v)$ of their complexity.  We choose a  representative in each
$S_i$, $i\in \{1,...,s\}$ and take for $X_i$ the corresponding semialgebraic set. Assume $X_{i_1}$ to be semialgebraically homeomorphic  to $X_{i_2}$ for some $i_1,\,i_2\in\{1,...,s\}$.
There is a couple of integers $(t_{i_1 i_2},u_{i_1 i_2})$ such that
there exists a semialgebraic  homeomorphism $f:X_{i_1}
\longrightarrow X_{i_2}$ whose graph belongs to
$\mathcal{A}(2n,t_{i_1 i_2},u_{i_1 i_2})$. Let $X$ and $Y$ be two semialgebraic sets 
belonging to $\mathcal{A}(n,p,q)$ such that they are semialgebraically  
 homeomorphic. Then there are $i_1$ and $i_2$ such that  $X=\mathcal{S}_a$ with  $a\in S_{i_1}$, 
  $Y=\mathcal{S}_b$ with  $b\in S_{i_2}$and  $X_{i_1}$,  $X_{i_2}$ are semialgebraically homeomorphic by $f$ as before. It follows that $X$ is semialgebraically 
 homeomorphic to $X_{i_1}$ by the trivialization homeomorphism, the same for $Y$ and $X_{i_2}$. We have more precisely: ${h_{i_1}}_{|X}:  X\longrightarrow  X_{i_1}$  defined by $(a,{h_{i_1}}_{|X}(x))=h_{i_1}(a, x).$ We get here that the complexity of this restriction is bounded by $(u,v)$ independently of $X$. And for $Y$, we have the semialgebraic homeomorphism  ${h_{i_2}}_{|Y}:  Y\longrightarrow  X_{i_2}$ defined by $(b,{h_{i_2}}_{|Y}(x))=h_{i_2}(b, x)$. Consequently this  homeomorphism has a complexity bounded by $(u,v)$, independently of $Y$. Hence we get  an homeomorphism from $X$ to $Y$ by $g= {h_{i_2}}_{|Y} ^{-1} \circ f \circ {h_{i_1}}_{|X}.$  The complexity of  $g$ is bounded by
$(t'_{i_1i_2},u'_{i_1i_2})$ independently of  $X$ and $Y$, since it is a composition of semialgebraic  homeomorphisms with complexity bounded independently of  $X$ and $Y$, and depends only on $i_1$ and  $i_2\in \{1,...,s\}$. Set
$$E=\{ (i,j)\in\{1,...,s\}^2 |\,\textrm{  $X_i$  and   $X_j$  \,  are semialgebraically  homeomorphic \}}.$$
This set is finite. Then, take
$(t, u)=\left( \max\limits_{(i,j)\in E}(t'_{i j}), \max\limits_{(i,j)\in E}(u'_{i j}) \right).$
\end{proof}

We can define the complexity of a semialgebraic cobordism.
\begin{defi}\rm  Let $(M,M_0,M_1)$ be a semialgebraic  cobordism such that the semialgebraic manifolds $M$, $M_0$ and $M_1$ have respective complexities $(t,u)$, $(t_0,u_0)$ and $(t_1,u_1)$. The \em{complexity} \rm  of  the cobordism $(M, M_0, M_1)$ is  $$(v,w)= (\max(t,t_0,t_1), \max(u,u_0,u_1)).$$

\end{defi}

The following theorem gives uniform bound for the h-cobordism theorem.  
\begin{teo}\label{bu}
Given $n,m\geq 6$, $(p,q)\in \N^2$, there exists  $(t,u)=\Psi_{HC}(n,m,p,q)$ in $\N^2$ such that for all simply connected semialgebraic h-cobordism $(M,M_0,M_1)$ in $\R^n$ of complexity at most $(p,q)$   and \rm $\mbox{dim}M=m$, \em{there exists a semialgebraic  homeomorphism $f:M\rightarrow M_0\times [0,1]$ whose graph $\Gamma_f\in \mathcal{A}(2n+1, t,
u).$}
\end{teo}

\begin{proof}
To prove the existence of the uniform bound $(t,u)$, we will first construct a set of parameters of semialgebraic  h-cobordisms in $\R^n$ with complexity at most $(p,q)$ and semialgebraically simply connected. We need  to translate the fact of being:
\begin{quote}
\virg{\it{a semialgebraic h-cobordism in $\R^n$ of complexity at most $(p,q)$ simply connected}},
\end{quote}
 into a first order formula of the theory of real closed fields.
 
Indeed, the fact that a semialgebraic subset of  $\R^n$ is a semialgebraic submanifold
of $\R^n$ of dimension $m$ can be said by  a first order formula of the theory of real closed fields
  (see Proposition \ref{variete} $(ii)$). Which implies that
the set of semialgebraic submanifolds of  $\R^n$
of dimension $m$ and with complexity at most $(p,q)$ is a semialgebraic subset of 
 $\mathcal{A}(n, p, q)$. Let us denote it by 
 $\mathcal{B}(n, m, p, q)$. So, it is defined by  a first order 
 formula of the theory of real closed fields  in
  in a uniform and  effective way.

The conditions which must satisfy a   triplet of semialgebraic manifolds\\
 $(M,M_0,M_1)$ to be a cobordism can be translated to a conjunction of 
 first order formulas of the theory of real closed fields with coefficients in
$\mathbb{Z}$.
Then the set of elements
$(a,\,b,\,c)\in \mathcal{B}(n, m, p, q)\times \mathcal{B}(n, m-1,
p, q)^2$ such that $(M_a, M_b, M_c)$ is a cobordism is a 
semialgebraic subset of  $\R^{3N + 3}$. This set parametrizes 
the semialgebraic cobordisms with complexity at most
 $(p, q)$ and we denote it by $\mathcal{C}ob(n, m, p, q).$ It is defined uniformly and effectively.\\
There is a  semialgebraic family 
$\mathcal{C}(n,m,p,q)\subset \mathcal{C}ob(n, m, p,
q)\times \R^n$ with two subfamilies $\mathcal{C}_0(n,m,p,q)\subset \mathcal{C}(n,m,p,q)$ and $\mathcal{C}_1(n,m,p,q)\subset \mathcal{C}(n,m,p,q)$ such that:
\begin{itemize}
\item For  every $b\in \mathcal{C}ob(n,m,p,q)$, the fiber 
$$\mathcal{C}_b(n,m,p,q)=\{x\in \R^n: (b,x)\in \mathcal{C}(n,m,p,q)\}$$ is a semialgebraic manifold of $\R^n$ of dimension $m$ of complexity at most $(p,q)$ with boundary the disjoint union of the fiber $\mathcal{C}_{0,b}(n,m,p,q)$ and $\mathcal{C}_{1,b}(n,m,p,q).$
\item For every semialgebraic cobordism $(M,M_0,M_1)$, $M\subset \R^n$ of dimension $m$ and complexity at most $(p,q)$, there exists $b\in \mathcal{C}ob(n,m,p,q)$ such that:
$$M=\mathcal{C}_b(n,m,p,q),\, \, \, M_0= \mathcal{C}_{0,b}(n,m,p,q),\,\,\, M_1= \mathcal{C}_{1,b}(n,m,p,q).$$
\end{itemize}
The families $\mathcal{C}(n,m,p,q)$, $\mathcal{C}_0(n,m,p,q)$ and $\mathcal{C}_1(n,m,p,q)$ are defined uniformly and effectively.
Consider the projection defined by:
\begin{align}
\Pi:\, &\mathcal{C}(n,m,p,q)\rightarrow \mathcal{C}ob(n, m, p,
q)
\nonumber \\
&\qquad \quad (a,x)\mapsto a. \nonumber
\end{align}
 Since $\Pi$ is a semialgebraic map, by  Hardt Theorem, there 
exists a finite semialgebraic  partition of  $\mathcal{C}ob(n,
m, p, q)= \bigcup\limits_{i =1}^{s}H_i$, compatible with the subfamilies\\ $\mathcal{C}_0(n,m,p,q)$ and $\mathcal{C}_1(n,m,p,q)$, such that  for all $i$ there exists a 
semialgebraic homeomorphism of trivialisation
 $\Pi_i:\Pi^{-1}(H_i)\rightarrow H_i\times C_i$ where $C_i= (C_i,{C_i}_0, {C_i}_1)$ is a semialgebraic h-cobordism. Assume $\Pi_i$ of complexity at most $(t_i,u_i)$.\\ Then, there is $J\subset \{1,...,s\}$  
  such that the union $\mathcal{H}cob(n, m, p, q)=\bigcup\limits_{j\in J}H_j$
  parametrizes the set of simply connected semialgebraic h-cobordisms  of complexity at most 
 $(p, q).$ This set is a semialgebraic.\\
We lose exactly here the effectiveness  because the problem of deciding  which semialgebraic cobordisms are  simply connected h-cobordisms is not effective (cf. [VKT]).\\
On the other hand the space of parameters $\mathcal{H}cob(n, m, p, q)$ is uniformly defined since the property of being semialgebraically  simply connected is invariant under extension of real closed fields (Proposition \ref{realcon}). Moreover, over $\R$  semialgebraic simple connectedness is the same as  topological simple connectedness (Proposition \ref{con}).\\
Let $(M,M_0,M_1)$ be a semialgebraic simply connected h-cobordism with a parameter $a\in \mathcal{H}cob(n, m, p, q)$, then there exists $j\in J$ such that $a\in H_j$. Hence, $\Pi_j|_M:M\rightarrow C_j$ is a semialgebraic homeomorphism with complexity at most $(t_j,u_j)$.
 ${C_j}_0\times [0,1]$ has a complexity bounded in terms$(p, q)$ in an effective way. Since $C_j$ and ${C_j}_0\times I$ are semialgebraically homeomorphic (Theorem \ref{hcob}),
then by  Proposition \ref{sdf}, there exists a couple of integers
$(v, w)$ which depends only on $n, p, q$ such that there exists a semialgebraic
homeomorphism $f_j: C_j \longrightarrow {C_j}_0\times I$
whose graph $\Gamma_f$ admits a complexity at most $(v,w)$. 
So we have the following semialgebraic homeomorphism:
$g_j=((\Pi_j|_{M_0})^{-1}\times id_I)\circ f_j\circ \Pi_j|_M: M \rightarrow M_0\times I.$
 We get that there exists  a bound on the complexity of $g_j$, write $(t'_j,u'_j)$, which depends only on $j$ and not on $(M,M_0,M_1)$.  Take  $$(t,u)=(\max_{j\in J}(t'_j),\max_{j\in J}(u'_j))$$   and this ends the proof.
\end{proof}
As we pointed out in the proof of the above theorem there is a precise point where we loose effectiveness even if we get uniform bounds. We shall look at this question in the next section.\\
We give now the semialgebraic  h-cobordism theorem over any real closed field.
Note that by \em{compact} \rm we mean \em{closed} \rm  and  \em{bounded}\rm.
\begin{teo}\label{hval}
Let $(M, M_0, M_1)$ be a semialgebraically simply connected semialgebraic h-cobordism defined over a real closed field
 $R$. If  \rm $\mbox{dim}\, M \geq 6$,
\em{then $M$ is semialgebraically homeomorphic to $M_0\times [0,1].$}
\end{teo}
\begin{proof} Fix $n$ the dimension of ambient space, $m\geq 6$ the dimension of semialgebraic h-cobordism and $(p,q)$ a bound on its complexity. By the above Theorem, there exists $(t,u)\in \N^2$ such that the following formula holds:

 $\Phi (n, m, p, q, t, u):=$
 \begin{quote}
 \virg{\it{For every semialgebraic h-cobordism $(M,
M_{0}, M_{1})$ in  $\R^n$ of complexity at most $(p, q)$
 simply  connected, there  exists  a semialgebraic  homeomorphism
$f: M \rightarrow M_0 \times [0,1]$  such that its 
graph $\Gamma_f \in \mathcal{A}(2n+1, t, u).$}}
\end{quote} 
We ask for  this sentence to be true  over any real closed field.  We can translate  the statement $\Phi (n, m, p, q, t, u)$ into a first order sentence of the theory of real closed fields.\\
Indeed, the space of parameters of semialgebraic h-cobordisms in $\R^n$ of complexity at most $(p,q)$ and semialgebraically simply connected of dimension $m$ is $\mathcal{H}cob(n, m, p, q)$ as constructed in the above Theorem.
  Denote by
$\mathcal{G}(n, p, q, t, u)$ the set of $(a, b, f)\in
\mathcal{A}(n, p, q)^2\times \mathcal{A}(2n+1, t, u)$ such that  $f:
\mathcal{S}_a \longrightarrow \mathcal{S}_b\times [0,1]$ is a semialgebraic  homeomorphism. 
The conditions that must be satisfied by $f$ to be a semialgebraic 
 homeomorphism, can be translated into a first order formula of the theory of real closed fields
with coefficients in $\mathbb{Z}$  in an  effective way (see the proof of Proposition \ref{homeo}). 
Consequently
$\mathcal{G}(n, p, q, t, u)$ is a semialgebraic set defined by a 
first order formula of the theory of real closed fields
with coefficients in $\mathbb{Z}$. We can now write the following statement:
\newline
$\Phi (n, m, p, q, t, u)$:
$$\virg{ \forall(a, b, c)\in \mathcal{H}cob(n, m, p, q)\,\exists f\in \mathcal{A}(2n+1, t, u) (a, b, f)\in \mathcal{G}(n, p, q, t, u)}.$$
The statement  $\Phi (n, m, p, q, t, u)$ as defined is a first order sentence of the theory of real closed fields
with coefficients in  $\mathbb{Z}$. Since $\R \models \Phi (n, m, p, q,
t, u)$, by
Tarski-Seidenberg Principle, for any real closed field $R$, one gets $
R\models \Phi (n, m, p, q, t, u)$. 

\end{proof}

\section{On non-effectiveness of semialgebraic\\ h-cobordism theorem}
We proved the existence of a uniform bound in the semialgebraic h-cobordism theorem. One the other hand one could expect, when working with semi-algebraic and compact PL objects, that bounds should be recursive in the sense of [Ma]. To be more precise, what we mean by effective is the following. A statement is effective if admits a uniform bound which is bounded by a recursive function.
It is not always the case. There are examples where uniform bounds  exist but  are not recursive. An example of this type can be found in  [ABB]. Namely: 

Let $K_{\Delta^m}$ be the standard triangulation of the standard simplex $\Delta^m$.
Let be $B=|K|$ a PL $m$-ball with $K$ a finite simplicial complex. 
By Standard Subdivision of $B$ we mean  a simplicial isomorphism $g:K'\rightarrow L$ where  $K'\lhd K$ and $L\lhd K_{\Delta^m}$.\\
Fixing $m$,  the authors proved the following:
\begin{quote}\it{
There exists $\Psi_{SS}(d)$ depending only on $d$ such that for any finite simplicial complex $K$ with at most $d$ vertices  such that $|K|$ is a 
PL m-ball, there exists $K'\lhd K$ with at most  $\Psi_{SS}(d)$ vertices, and a simplicial isomorphism of $K'$ with $L$ a subdivision of $K_{\Delta^m}$.}
\end{quote}
But the preceding uniform bound is not recursive
\begin{teo} ([ABB], Corollary 2.18)\label{abb}\\
For $m\geq 5$, $\Psi_{SS}$ cannot be bounded by a recursive function.
\end{teo}
We can now prove the non-effectiveness of the semialgebraic h-cobordism theorem.
\begin{teo} For $m\geq 6$, the uniform bound $\Psi_{HC}$ of Theorem \ref{bu} cannot be bounded by a recursive function.
\end{teo}
\begin{proof}
Let be $B$ a PL $m$-ball with $m\geq 6$. Assume $B$ triangulated by a finite simplicial complex $K=\{\sigma_1,...,\sigma_n\}$ with at most $d$ vertices.
 Let $\tau$ be  an $m$-dimensional simplex in $B$ such that $|\tau|\cap \partial B =\emptyset$ (subdivide $K$ if necessary).
It is clear that  $(\overline{B\setminus \tau}, \partial \tau,\partial B)$ is a PL simply connected h-cobordism. The complexity of this h-cobordism is bounded by a recursive function in terms of $d$.\\
 Assume  $\Psi_{HC}$ to be recursive. Then there is  a semialgebraic homeomorphism:
 $h: \overline{B\setminus \tau}\rightarrow \partial \tau \times I$ whose  complexity is recursively  in terms of $d$.  We can attach $\tau$ to $\partial \tau \times I$, identifying $\partial \tau \subset \tau$ with $\partial \tau\times \{0\}$. Then extending  $h$ over $B$ by the identity on $\tau$, we get a semialgebraic homeomorphism  $$h': B\rightarrow \tau \bigcup\limits_{\partial \tau} (\partial \tau \times I) $$ with complexity bounded by a recursive function $\Phi(d)$. Let $\tilde{x}= (x_0,...,x_m)$ with $x_i\geq 0$ and $\sum x_i=1$  be a point of $\tau$  with its barycentric coordinates. Let $\tilde{b}= (\frac{1}{m+1},...,\frac{1}{m+1})$ the barycenter of standard $m$- simplex. Now consider the following PL homeomorphism $g: \tau \bigcup\limits_{\partial \tau} (\partial \tau \times I)\rightarrow \Delta_m$ defined by $$g(x)=\left\{
\begin{array}{cccl}
 \frac{1}{2}\tilde{b}+ \frac{1}{2}\tilde{x} &&& \mbox{if} \,  x\in \tau\\
\frac{1-\lambda}{2}\tilde{b}+\frac{1+\lambda}{2}\tilde{x} , &&& \mbox{if}\,  (x,\lambda)\in \partial \tau \times I.
 \end{array}
\right. $$ So we get  a semialgebraic  homeomorphism $f: B\rightarrow \Delta_m$, given by $f= g\circ h'$, of complexity  bounded by a recursive function $\Theta(d)$ in term of $d$. Since $f$ is a semialgebraic homeomorphism of complexity bounded by $\Theta(d)$ between two compact simplicial complexes with at most $d$ vertices, the effective semialgebraic Hauptvermutung (see [Co2]), implies that there exists  a recursive function $\chi$ in terms of $\Theta(d)$ and $d$ (so just in term of $d$) such that there is a simplicial isomorphism between the subdivisions of these simplicial complexes  with at most  $\chi(d)$ vertices. This is in contradiction with Theorem \ref{abb}.
\end{proof}

\section{Nash h-cobordism theorem over any real closed field}

Now we consider  Nash manifolds.
\begin{defi}
  Let $ M$, $M_0$, $M_1$  be  compact Nash manifolds such that:
 $\partial M = M_0\cup M_1$
and   $ M_0\cap M_1= \emptyset .$
Then, the triplet $ ( M, M_0, M_1) $ is called a \em{Nash cobordism
}\rm.\\
  A Nash cobordism $( M, M_0, M_1) $ is said to be a \em{Nash h-cobordism} \rm  if the inclusions
$M_0 \hookrightarrow M$ and
 $M_1 \hookrightarrow M$ are  semialgebraic homotopy equivalences, that is, the deformation retractions are semialgebraic.
\end{defi}

\begin{teo} ( \rm Nash h-cobordism theorem)\label{hcobn}
\newline
\em{Let $( M, M_0, M_1) $ be a simply connected Nash h-cobordism .
If} \rm $\mbox{dim} M\geqq 6$  \em{then $M$ is Nash diffeomorphic
to $M_0 \times [0,1].$}
\end{teo}
This theorem is an easy consequence of differentiable h-cobordism theorem quoted in the introduction and of the following Nash approximation theorem:
\begin{teo} ([S], Theorem VI.2.2, p.202) \label{appron}\\
Let $L_1$, $L_2$ be  compact Nash manifold possibly with boundary, and $M_1$, $M_2$ their interior.
The following conditions are equivalent.\\
 (i)  $L_1$ and $L_2$ are  $C^1$ diffeomorphic.\\
 (ii)  $L_1$ and $L_2$ are Nash diffeomorphic.\\
 (iii)  $M_1$ and $M_2$ are  Nash diffeomorphic.\\
\end{teo}

We state two results which give an analogue of Hardt theorem for Nash manifolds with boundaries.
\begin{teo} \label{sub}
Let $B\subset R^p$ be a semialgebraic set, let $X$ be a semialgebraic subset of  $R^n\times B$ such that for every  $b\in B$, $X_b$ is a Nash submanifold of $R^n$.Then there is a  stratification $B=\cup_{i\in I}M^i$ of $B$ into  a finite number of  Nash manifolds, such that for any $i\in I$, $X_{|M^i}$ is  Nash manifold and the projection $X_{|M^i}\rightarrow M^i$ is a  submersion. If, moreover, $X_b$ is compact for every $b\in B$, we can ask this  submersion to be proper.
\end{teo}
\begin{proof}
See [CS, Corollary 2.3].
\end{proof}
\begin{teo} \label{subbord}
Let $M\subset R^{m'}$ be a  Nash submanifold of dimension $m$ possibly with boundary $\partial M$.
Let $\varpi: M\rightarrow R^k$, $k>0$ be a proper onto Nash submersion such that $\varpi_{|\partial M}$ is onto submersion. Then there exists a  Nash diffeomorphism 
$\varphi=(\varphi',\varpi):(M, \partial M)\rightarrow (M\cap \varpi^{-1}(0), \partial M\cap \varpi^{-1}(0))\times R^k.$
 \end{teo}
 \begin{proof} See [FKS, Theorem I].
 \end{proof}
We can now  formulate a Nash triviality in family of Nash manifolds with boundaries.
\begin{teo}\label{famille}
Let $B$ be  a semialgebraic set and   $\Pi: R^n\times
B\rightarrow B$ be the  projection on $B$. Let $X$ be a 
semialgebraic subset of $R^n\times B$ such that for all $b\in B$,
$$X_b = \{x\in R^n: \, (x,b)\in X \}$$
is a compact Nash manifold in $R^n$ with boundary. Then there  exists a finite partition of 
 $B$ in $B=\bigcup_{i\in I}\, M^i$ where
$M^i$ are Nash manifolds, and for each  $i\in I$ there is an affine Nash manifold
 $F^i\subset R^n$ with boundary and a Nash diffeomorphism which trivializes $\Pi_{|X_{|M_i}}$
$$h^i: F^i\times M^i \rightarrow X\cap \Pi^{-1}(M^i)$$
compatible with the boundary.

\end{teo}
\begin{proof} Set $Y= \{(x,b): x\in\partial X_b\}$.
By Theorem \ref{sub}, one can prove that there exists a finite Nash partition of $B$ into 
$B=\bigcup\limits_{i=1}^r B_i$  such that for every $i$, $B_i$, $X_{|B_i}$ and $Y_{|B_i}$ are Nash submanifolds and the projections $\pi_{|X_{|B_i}}:X_{|B_i}\rightarrow B_i$ and $\pi_{|Y_{|B_i}}:Y_{|B_i}\rightarrow B_i$ are proper onto Nash submersions.
 For any $i$ , there exists a partition of  $B_i$, by [BCR, Proposition 2.9.10, p.57],  into $B_i= \bigcup\limits_{j=i}^t S_{i_j}$ such that $S_{i_j}$ is Nash diffeomorphic to $]0,1[^{k_{i_j}}$. However, it is clear that $]0,1[^{k_{i_j}}$ is Nash diffeomorphic to $R^{k_{i_j}}$. So we get the proper onto Nash submersions 
$X_{|S_{i_j}}\rightarrow R^{k_{i_j}}$ and $Y_{|S_{i_j}}\rightarrow R^{k_{i_j}}$.
By  Theorem \ref{subbord} applied to $X_{|S_{i_j}}$ with its boundary $\partial X_{|S_{i_j}}= Y_{|S_{i_j}}$, there exists a Nash submanifold with boundary $F_{i_j}$  of $R^n$ such that:
$(X_{|S_{i_j}}; {Y}_{|S_{i_j}})\cong(F_{i_j};\partial F_{i_j})\times R^{k_{i_j}}.$
This is equivalent to:
$(X_{|S_{i_j}}; Y_{|S_{i_j}}) \cong(F_{i_j} ;\partial F_{i_j})\times S_{i_j}.$
Which ends the proof.
\end{proof}
By a result of Ramanakoraisina (cf.\, [R, Proposition 3.5]), given $m,p,q$ there is an integer $l$ such that for any semialgebraic map $f$ of complexity $(p,q)$, over an open  subset of $\R^m$,  $f$ is Nash if and only if $f$ is $C^l$. Moreover, the integer $l$ is recursive in terms of $m,p,q$ (cf.[CS, Lemma 5.1]). From this we get the following proposition.
\begin{pro} \label{nvariete}
 Let $S$ and $T$ be semialgebraic subsets of $\R^n$ such that $T\subset S$. The following statements can be translated into a first order  formula of the theory of real closed fields:\\
 (i)  \virg{$S$ is a Nash submanifold of  $\R^n$ of dimension $m$}\\
 (ii) \virg{$S$  is a Nash submanifold of $\R^n$ of dimension $m$,  with boundary 
 the set $T$} 
\end{pro}
\begin{proof} 
(i) Just use the notification above and  the proof of Proposition \ref{variete}.\\
(ii) Use the notification above and the proof of Proposition \ref{bord}.
\end{proof}
In a similar way to Proposition \ref{homeo}, one has the following result in the Nash case.
\begin{pro}\label{diffeo}
Let $R$ and $K$ be two real closed fields such that $K$ is a real closed extension of $R$. Let $X$ and $Y$ be two  Nash submanifold
$R^n$. The Nash manifolds $X$ and $Y$ are Nash diffeomorphic if and only if $X_K$ and $Y_K$ are Nash diffeomorphic.
\end{pro}
\begin{proof}
The proof is similar to the proof of  Proposition \ref{homeo} adding the property that a semialgebraic map must  satisfy to be a Nash map which can be translated in a first order formula of the theory of real closed fields (cf.[R, Proposition 3.5]). This ends the proof.
\end{proof}
 We prove here the analogue of  Proposition \ref{sdf} for Nash manifolds.
\begin{pro}\label{born}
Given the integers  $n$, $p$ and $q$, there exists a couple
of integers $(t,u)$ such that for all  couples of Nash submanifolds of
$R^n$ of  complexity at most $(p,q)$ and  Nash diffeomorphic,
 there exists a Nash diffeomorphism $f$between them, whose graph 
$\Gamma_f$ admits a  complexity at most $(t,u)$.
\end{pro}
\begin{proof}
Let us notice that all Nash submanifolds of  $R^n$ of complexity 
at most $(p,q)$ belong to $\mathcal{A}(n,p,q)$. We have shown in the proof of 
 Proposition \ref{sdf} that the space of 
parameters  $\mathcal{A}(n,p,q)$  is a semialgebraic set.
The set of  parameters of Nash submanifolds of  $R^n$ of 
dimension $m$ of complexity at most $(p,q)$ is included in $\mathcal{A}(n,p,q)$  and is a semialgebraic subset
of $\mathcal{A}(n,p,q)$. Indeed, the condition for a 
semialgebraic subset to be a Nash manifold  can be translated uniformly and effectively in a first order formula of the theory of real closed fields 
(Proposition \ref{nvariete}.(i)). Let us denote this set by
$\mathcal{N}(n,m,p,q)$. By  the Nash Triviality  theorem 
mentioned above, there exists a finite Nash stratification 
$\mathcal{N}(n,m,p,q)=\bigcup\limits_{i=1}^s M^i$ of
$\mathcal{N}(n,m,p,q)$ into Nash manifolds $M^i$ such 
that the Nash manifolds parametrized by points in $M^i$ are Nash
diffeomorphic. The remainning part of the proof is similar to the 
end  of the proof of Proposition \ref{sdf}.
\end{proof}

Now, the existence of  a uniform bound for  Nash h-cobordism Theorem, can be deduced in the same way as for the semialgebraic case
\begin{teo}\label{bun}
 Given $n,m\geq 6$, $(p,q)\in \N^2$, there exists  $(t,u)\in \N^2$ such that for all Nash h-cobordism $(M,M_0,M_1)$ in $\R^n$ of complexity at most $(p,q)$ simply connected and \rm $\mbox{dim}M=m$, \em{there exists a Nash diffeomorphism $f:M\rightarrow M_0\times [0,1]$ such that its graph $\Gamma_f\in \mathcal{A}(2n+1, t,
u).$}
\end{teo}

 The validity over any real closed field  follows as consequence of the above theorem.
\begin{teo}\label{hvaln}
Let $(M, M_0, M_1)$ be a Nash  h-cobordism, semialgebraically simply connected defined over a real closed field
 $R$. If \rm $\mbox{dim}\, M \geq 6$, \em{then $M$
is Nash diffeomorphic to $M_0\times [0,1].$}
\end{teo}

\section{Semialgebraic and Nash s-cobordism theorems}
We first  define the notion of simple homotopy equivalence.
\begin{defi} \rm
 Let $P$ and $Q$ be two polyhedra such that $Q\subset P$. Let
$B^n$ a PL n-ball. If $P= Q\cup B^n$ , $Q\cap B^n \subset \partial B^n$ and $Q\cap B^n$ is a PL $(n-1)$-ball, we say
that $Q$ is obtained from $P$ by an  \em{ elementary collapse.}
\rm We denote it by $P\Downarrow Q.$ We say also that there is 
\em{elementary expansion } \rm of $Q$ in $P$.

We say that  $P$ collapse on $Q$ and we denote it by $P\searrow Q$  if there is a finite sequence of elementary collapses  $P=P_0 \Downarrow P_1
\Downarrow ...\Downarrow P_m =Q.$ We will say also that there is an expansion of $Q$ to $P$ and we denote it by $Q\nearrow P$.
\end{defi}
\begin{defi}\rm
A polyhedron $P$ is said to be \em{simply homotopic } \rm to another polyhedron $Q$ if $P$  is obtained from $Q$ by a sequence  of  collapses and  expansions of $Q$
$$P= P_0\searrow P_1\nearrow P_2\searrow...\searrow P_n= Q \,\mbox{rel}\, Q$$ where rel $Q$ means that during the collapses and expansions operations $Q$  remains unchanged. We will say: $P$ is \em{s-homotopic} \rm to
$Q$.
\end{defi}

Now we define the notion of semialgebraic  simple homotopy equivalence.
\begin{defi} \label{colapse}\rm
Let $X$ and  $Y$ be two semialgebraic sets such that $Y\subset
X$.  We say that  $X$   \em{elementary semialgebraically  collapses} \rm on $Y$, and we write $ X\Downarrow Y$, if there is a semialgebraic map $ f: I^n \longrightarrow X$ such that
\begin{itemize}
\item $f|_{(0,1]\times I^{n-1}}$ is an embedding
\item $f(0\times I^{n-1}) \subset Y$,  $f((0,1]\times I^{n-1}) \cap
Y = \emptyset$ and $X = Y\cup f(I^n)$
\end{itemize}
We say that  $X$  semialgebraically collapses on  $Y$, and write
  $X\searrow Y$, if there is a finite sequence of elementary semialgebraic  collapses $X=X_0 \Downarrow X_1 \Downarrow
...\Downarrow X_m =Y.$
\end{defi}

\begin{remark}\label{collapse} \rm
A PL elementary collapse of compact polyhedra is in particular a semialgebraic collapse. A kind of converse is proved in next lemma.
\end{remark}
In the following definition, the notation  \virg{rel $Y$} means that during the collapse and expansion operations, $Y$ remains unchanged.
\begin{defi}\label{torsion}\rm
Let $Y\subset X$ be compact semialgebraic sets.
 The semialgebraic set $X$ is  \em{semialgebraically s-homotopic  } \rm to
 the semialgebraic set $Y$ if and only if there is a sequence of semialgebraic collapses and expansions of $X$ on $Y$ rel $Y$.
\end{defi}

\begin{lem}\label{csapl}
Let  $Y\subset X$ be two  compact semialgebraic sets and assume  $ X\Downarrow Y$. Take a triangulation $|K|\rightarrow X$ compatible with $Y$ and put $|K'|=h^{-1}(Y)$. Then  $|K|
\searrow |K'|$.
\end{lem}
\begin{proof}
By hypothesis $X\Downarrow Y$.
 This means that there is a semialgebraic map $f: I^{n}\longrightarrow X $ satisfying the properties of Definition \ref{colapse}.
Set $g =f_{|0\times I^{n-1}}$. We may, consider $X$ as a  mapping cylinder $M_g$ of
$g$, setting
$ X = M_g = I^{n}\cup Y / {(0,x)\sim g(0,x)}.$
A semialgebraic  triangulation $|K|$ of $X$, compatible with $Y$ and
$cl( M_g\setminus Y)$, induces a semialgebraic  triangulation $|K'|$ of $Y$ and another $|K"|$ of $cl( M_g\setminus Y)$. This means that $K'\subset K$ and $K"\subset K$. It follows that $
|K| =|K"| \cup |K'|.$ We can construct a projection
$\pi :|K"| \longrightarrow [0,1]$  such that  $\pi ^{-1}(0) = |K"|
\cap |K'|,$ in the following way. Let $x$ and $y$ be in $I^n$. One
defines an equivalence relation "$x \approx y$"  if and only if
 $f(x)= f(y)$. The  map induced by $f$,  say $f'$,
defined by 
\begin{align}
I^n/\approx & \longrightarrow f(I^n) \nonumber\\
[t]&\mapsto f'([t])= f(t) \nonumber
\end{align}
 where $t=(t_1,...,t_n)\in I^n$,
is  a homeomorphism for  the quotient  topology. The 
semialgebraic triangulation of $X$ induces a semialgebraic 
homeomorphism $k: |K"|
\longrightarrow f(I^n)$. Let be $p: I\times I^{n-1} \longrightarrow I$ such that
 $p(t_1,...,t_n)= t_1$ for all $(t_1,...,t_n)\in I^n$ and $p': I^n / \approx \rightarrow I$ the  induced continuous map. Consider the map $\pi$
as follows
$ \pi = p'\circ f'^{-1}\circ k: |K"| \longrightarrow [0,1].$
There exists $0<\epsilon < 1$ such that $\pi
^{-1}([0,\epsilon])$ contains all vertices of the simplices 
that have at least one face in $\pi ^{-1}(0)$ making some 
subdivisions of $K"$ if necessary. We observe that:
$\pi^{-1}([0,\epsilon]) = k^{-1}(f([0, \epsilon]\times I^{n-1} )).$ The semialgebraic subset
 $f([0, \epsilon]\times I^{n-1})$ is a semialgebraic compact neighbourhood of
$f(0 \times I^{n-1})$ in the semialgebraic set $f(I^n).$ It follows that $\pi
^{-1}([0,\epsilon]) = U$ is a regular neighbourhood $\pi ^{-1}(0)$
in $|K"|$, asking that $k$ is compatible with
$f([0, \epsilon]\times I^{n-1})$. Since $ cl( |K"|\setminus U)$
is PL homeomorphic to $I^{n}$, $ |K"|$ PL collapses  on $U$. Furthermore, we have that $U$ PL collapses  on $\pi ^{-1}(0)$, since it is a regular neighbourhood of $\pi ^{-1}(0)$ cf. ([RS], Corollary 3.30). We get that
$|K|$ PL  collapses on $|K'|.$ This closes the proof.
\end{proof}
\begin{coro}\label{tau}
Let $M$ and $M_0$ be two compact semialgebraic manifolds with $M_0$ a deformation retraction of 
$M$. Then, $M$ collapses
semialgebraically on $M_0$ if and only if $|K|$ PL collapses 
on $|K_0|$, where $|K|$ is a semialgebraic triangulation of $M$ compatible with $M_0$
and $|K_0|$ the induced  triangulation of $M_0$.
\end{coro}
\begin{proof}
The proof follows using Lemma \ref{csapl} and  Definition
\ref{torsion}.
\end{proof}
We get now a semialgebraic version of the s-cobordism theorem.
\begin{teo}
Let $(M, M_0, M_1)$ be a connected semialgebraic  h-cobordism with $\mbox{dim}\, M \geq 6$. Then, $M$ and $M_0 \times [0,1]$ are semialgebraically  homeomorphic if and only if
$M$ is semialgebraically s-homotopic   to $M_0.$
\end{teo}
\begin{proof}
It is an easy consequence of the classical PL s-cobordism  theorem using Corollary \ref{tau}.
\end{proof}
To close this section, we prove the Nash version of s-cobordism theorem.
\begin{teo}
Let $(M, M_0, M_1)$ be a connected Nash  h-cobordism with
$\mbox{dim}\, M \geq 6$. then $M$ and $M_0 \times [0,1]$ are Nash
diffeomorphic if and only if $M$ is semialgebraically s-homotopic   to $M_0.$
\end{teo}

\begin{proof}
 $\Rightarrow$) Assume that $M$ is Nash diffeomorphic to $M_0 \times [0,1]$ .
Since $M$ and $M_0 \times [0,1]$ are Nash manifolds, they are in particular  smooth and diffeomorphic.
By the smooth s-cobordism theorem, one has
 $M$ is  s-homotopic  to $M_0.$ This implies that $M$ is semialgebraically s-homotopic
 to $M_0$.
\newline
 $\Leftarrow$) Conversely, assume that $M$ is semialgebraically s-homotopic  to $M_0.$  This implies in particular that $M$ is s-homotopic  to
$M_0.$ By  the smooth s-cobordism theorem (cf.[K]), one gets $M$ diffeomorphic to $M_0 \times [0,1].$ It follows by  Theorem \ref{appron} that $M$ is Nash diffeomorphic to
$M_0 \times [0,1]$. This closes the proof.
\end{proof}
\section{Semialgebraic and Nash s-cobordism theorems over any real closed field}

In this section we prove that the  semialgebraic and Nash 
s-cobordism theorems hold over any real closed field. To make short, we will just write \virg{sas-homotopic} instead of \virg{semialgebraically s-homotopic}. Since one implication in the s-cobordism theorem is obvious, in the following we  will only consider the other one.

First we see that sas-homotopy is preserved in a real closed extension.
\begin{pro}
Let $X$ and $Y$ be two semialgebraic subset of  $R^r$ and $K$ a closed real extension of $R$. Then,
$X \Downarrow Y$ algebraically if and only if $X_K \Downarrow Y_K$ semialgebraically.
\end{pro}
\begin{proof} Use Tarski-Seidenberg Principle.
\end{proof}
\begin{pro}\label{omo}
Let $X$ and $Y$ be two semialgebraic subset of $R^r$ and $K$ a real closed extension of  $R$. Then,
$X$ is sas-homotopic to $Y$ if and only if $X_K$ is sas-homotopic to $Y_K$.
\end{pro}
\begin{proof}
Set:
$X= \{x\in R^r: \phi(a,x) \},$
$Y=\{x\in R^r: \psi(b,x) \}$ where $\phi(a,x)$ and $\psi(b,x)$ are first order formulas of the theory of real closed fields
with parameters
 $a\in R^m$ and $b\in R^{m'}$.\\
$\Rightarrow)$ $X$ is sas-homotopic to $Y$ 
if there is a sequence of semialgebraic collapses and expansions: 
$X=X_0\searrow X_1\nearrow X_2\searrow...\searrow X_s= Y\, \mbox{rel} \,Y.$
Let  $i$ be such that  $X_i\searrow X_{i+1}$. It is equivalent to say that there 
exists  a finite sequence of elementary semialgebraic  collapses. By the above Proposition, one has ${X_i}_K \searrow {X_{i+1}}_K$.\\ If on the other hand  $X_i\nearrow X_{i+1}$, it is the same to say that 
$X_{i+1}\searrow X_i$. We deduce in the same way that $X$ is sas-homotopic to $Y$.\\
$\Leftarrow)$ $X_K$ is sas-homotopic $Y_K$ if there  exists a sequence of semialgebraic  collapses and expansions: $X_K=X_0\searrow X_1\nearrow X_2\searrow...\searrow X_s= Y_K\, \mbox{rel} \,Y_K.$
Assume that $X_i=\{x\in K^r: \phi_i(a_i,x)\}$ with $a_i\in K^{m_i}$, where $\phi_i(a_i,x)$ is a first order  formula of the theory of real closed fields for  $i=0,...,s$  and  
$\phi_0(a_0,x)\}=\phi(a,x)$, $a_0=a$, $m_0=m$, $\phi_s(a_s,x)=\psi(b,x)$, $a_s=b$ and $m_s=m'$.\\
Let  $i\in\{0,...,s\}$ be such that $X_i\searrow X_{i+1}$. This implies that there exists a 
sequence of elementary semialgebraic  collapses  $X_i={X_i}_0\Downarrow {X_i}_1 \Downarrow...\Downarrow {X_i}_{k_i} = X_{i+1}.$
Set ${X_i}_j= \{x\in K^r: {\phi_i}_j({a_i}_j,x)\}$. The fact that  ${X_i}_j\Downarrow{X_i}_{j+1}$ can be translated in a first order sentence of the theory of real closed fields, denote it by: ${\beta_i}_j({a_i}_j, {a_i}_{j+1},{c_i}_j)$ where ${c_i}_j$ is a  parameter which define the graph of the semialgebraic map which define the elementary  semialgebraic  collapse. Then the fact that  $X_i\searrow X_{i+1}$ can be translated into a first order sentence of the theory of real closed fields as follows:
$ \bigwedge_{j=0}^{k_i} {\beta_i}_j({a_i}_j, {a_i}_{j+1},{c_i}_j).$
Now assume that  $X_i\nearrow X_{i+1}$. This is equivalent to $X_{i+1}\searrow X_i$. By the same techniques we construct a first order formula of the theory of real closed fields as follows:
$\bigwedge_{j=0}^{k_i} {\beta_i}'_j({a_{i+1}}_j, {a_{i+1}}_{j+1},{c_{i+1}}_j).$
Then the fact that  $X_K=X_0\searrow X_1\nearrow X_2\searrow...\searrow X_s= Y_K\, \mbox{rel} \,Y_K$ can be translated into a first order sentence of the theory of real closed fields as following:
$$(\bigwedge_{i=0}^{s}(\bigwedge_{j=0}^{k_i} {\beta_i}_j({a_i}_j, {a_i}_{j+1},{c_i}_j)\bigvee(\bigwedge_{j=0}^{k_{i+1}} {\beta_i}'_j({a_{i+1}}_j, {a_{i+1}}_{j+1},{c_{i+1}}_j)))\bigwedge(\bigwedge_{i,j} \lambda({a_i}_j, b)),$$
where $\lambda({a_i}_j, b)$ translate the fact that $Y_k\subset {X_i}_j$ and the collapses let fix pointwise $Y_k$ for all $i$ and  $j$. Let us denote this sentence by
$\Phi(({a_i}_j),({c_i}_j)),$ where ${a_0}_0= a$ and ${a_s}_{k_s}= b$. By hypothesis one gets 
 $K\models \Phi(({a_i}_j),({c_i}_j)).$ 
Some parameters are already in $R$. Take the parameters  ${a_l}_t$ and ${c_u}_v$ which are not. Omitting the  parameters defined over  $R$ in the  formula $\Phi$ (to make short), one gets:
$K\models \exists{y_l}_t \exists {z_u}_v \Phi({y_l}_t, {z_u}_v),$ where the quantification runs over the  variables with indices $l_t,\, u_v$ in $\Phi$.
So, by Tarski-Seidenberg Principle, we have:
$R\models \exists{y_l}_t \exists {z_u}_v \Phi({y_l}_t, {z_u}_v).$
This means that there exists  a sequence 
$X_K=X_0\searrow X_1\nearrow X_2\searrow...\searrow X_s= Y_K\, \mbox{rel} \,Y_K$ 
with all $X_i$ and semialgebraic  collapses defined by first order formulas of the 
theory of real closed fields with coefficients in $R$. This  implies that  $X$ is sas-homotopic to $Y$.
\end{proof}

We  prove that there is a uniform bound on the complexity of the semialgebraic homeomorphism in term of the complexity of the semialgebraic h-cobordism in the s-cobordism theorem. 
\begin{teo}\label{bus}
Given $n,m\geq 6$, $(p,q)\in \N^2$, there exists  $(t,u)\in \N^2$ such that for all semialgebraic h-cobordism $(M,M_0,M_1)$ in $\R^n$ of complexity at most $(p,q)$ with $M$ sas-homotopic to $M_0$ and \rm $\mbox{dim}M =m$, \em{there exists a semialgebraic  homeomorphism $f:M\rightarrow M_0\times I$ such that its graph $\Gamma_f\in \mathcal{A}(2n+1, t,
u).$}
\end{teo}
\begin{proof}
We follow the proof of Theorem \ref{bu}, instead of  choosing, in the partition of $Cob(n,m,p,q)$,  the semialgebraic h-cobordism semialgebraically simply  connected, we select the semialgebraic h-cobordism $(M, M_0, M_1)$ such that $M$ is sas-homotopic  to $M_0$ and connected.\\
This gives us  a semialgebraic set of parameters of the semialgebraic h-cobordism $(M,M_0,M_1)$ in $\R^n$ of complexity at most $(p,q)$ with $M$ sas-homotopic to $M_0$, which we denote by $Scob(n,m,p,q)$. The semialgebraic set $Scob(n,m,p,q)$ is defined uniformly by Proposition \ref{omo}.\\
Then we change  $\mathcal{H}cob(n,m,p,q)$ by $Scob(n,m,p,q)$ in the remaining part of the proof of Theorem \ref{bu}.
\end{proof}
Here is the semialgebraic s-cobordism theorem over any real closed field.
\begin{teo}
Let $(M, M_0, M_1)$ be a semialgebraic  h-cobordism defined over a real closed field $R$ such that $M$ sas-homotopic to $M_0$. If \rm $\mbox{dim}\, M \geq 6$, \em{then $M$
is semialgebraically homeomorphic to $M_0\times I.$}
\end{teo}
\begin{proof} Fixing a bound on the complexity of the semialgebraic h-cobordism,
with Theorem \ref{bus}, the proof is the same as the proof of Theorem \ref{hval}.
\end{proof}

We get in the same way the Nash version of these theorems.
\begin{teo}\label{busn}
Given $n,m\geq 6$, $(p,q)\in \N^2$, there exists  $(t,u)\in \N^2$ such that for all Nash h-cobordism $(M,M_0,M_1)$ in $\R^n$ of complexity at most $(p,q)$ with $M$ sas-homotopic to $M_0$ and \rm $\mbox{dim}M=m$, \em{there exists a Nash  diffeomorphism $f:M\rightarrow M_0\times I$ such that its graph $\Gamma_f\in \mathcal{S}(2n+1, t,
u).$}
\end{teo}
\begin{proof}
The proof is  the same as the proof of Theorem \ref{bus}. 
\end{proof}
We close this section with the Nash s-cobordism Theorem over any real closed field.
\begin{teo}
Let $(M, M_0, M_1)$ be a Nash  h-cobordism defined over a real closed field $R$ such that $M$ sas-homotopic to $M_0$. If \rm $\mbox{dim}\, M \geq 6$, \em{then $M$
is Nash diffeomorphic to $M_0\times [0,1].$}
\end{teo}
\begin{proof} Fixing a bound on the complexity of the Nash h-cobordism,
with Theorem \ref{busn}, the proof is the same as the proof of Theorem \ref{hvaln}.
\end{proof}

IRMAR (UMR CNRS 6625), Université de Rennes 1, Campus de Beaulieu, 35042 Rennes cedex, France\\
 email: \texttt{kartoue.demdah@univ-rennes1.fr}\\\\\
and \\ \\
 Dipartimento di Matematica \virg{Leonida Tonelli}, 
Largo Bruno Pontecorvo 5,
56127, Pisa
Italy.

\end{document}